\documentclass[12pt]{amsart}

\usepackage{graphicx}
\usepackage{amsmath,amsthm,euscript}
\usepackage{amsfonts}
\usepackage{amssymb}
\usepackage{a4wide}

\newtheorem{thm}{Theorem}[section]
\newtheorem{lem}[thm]{Lemma}

\newtheorem{prop}[thm]{Proposition}

\newtheorem{conj}[thm]{Conjecture}

\theoremstyle{definition}

\newtheorem{quest}{Question}

\begin{document}
\title{Rings of differential operators on curves}
\author{Jason P. Bell}
\thanks{The research of the first-named author was supported by NSERC grant 31-611456.}  

\address{Jason Bell\\
Department of Mathematics\\
Simon Fraser University\\
Burnaby, BC, V5A 1S6\\ CANADA}
\email{jpb@math.sfu.ca}
\author{Agata Smoktunowicz}
\thanks{The research of the second-named
author was supported by Grant No. EPSRC
 EP/D071674/1.}
\address{Agata Smoktunowicz\\
Maxwell Institute for Mathematical Sciences\\
School of Mathematics, University of Edinburgh\\
James Clerk Maxwell Building, King's Buildings, Mayfield Road\\
Edinburgh EH9 3JZ, Scotland\\
 UK}
 \email{A.Smoktunowicz@ed.ac.uk}

\bibliographystyle{plain}

\begin{abstract}
Let $k$ be an algebraically closed field of characteristic $0$ and let $A$ be a finitely generated $k$-algebra that is a domain whose Gelfand-Kirillov dimension is in $[2,3)$.  We show that if $A$ has a nonzero locally nilpotent derivation then $A$ has quadratic growth.  In addition to this, we show that $A$ either satisfies a polynomial identity or $A$ is isomorphic to a subalgebra of $\mathcal{D}(X)$, the ring of differential operators on an irreducible smooth affine curve $X$, and $A$ is birationally isomorphic to $\mathcal{D}(X)$.
\end{abstract}
\maketitle
\section{Introduction}
Given a field $k$ and a $k$-algebra $A$, a \emph{derivation} $\delta$ of $A$ is a map from $A$ to itself satisfying the Leibniz rule: $\delta(ab)=a\delta(b)+b\delta(a)$ for $a,b\in A$.  
A derivation $\delta$ of $A$ is called \emph{locally nilpotent} if for each $a\in A$ there is some natural number $n$ such that $\delta^n(a)=0$.

The study of locally nilpotent derivations has seen wide use in affine algebraic geometry.  Makar-Limanov \cite{ML} answered a problem of Koras and Russell \cite{KR} by showing that the smooth affine threefold given by the zero set of $X(1+XY)+Z^2+T^3$ in $\mathbb{C}^4$ is not isomorphic to $\mathbb{C}^3$ by considering the intersection of the kernels of all locally nilpotent derivations for both rings.  This invariant is now called the \emph{Makar-Limanov} invariant.  Derksen \cite{D} gave an alternative proof of Makar-Limanov's result by constructing another invariant---the subalgebra generated by all kernels of locally nilpotent derivations---that is now called the \emph{Derksen invariant}.

Since their initial applications, both the Makar-Limanov and Derksen invariants have seen widespread use in the study of the Zariski cancellation problem, which asks whether it follows that two affine complex varieties $X$ and $Y$ with the property that $X\times \mathbb{C}\cong Y\times \mathbb{C}$ are isomorphic.  

It is natural to ask whether the study of locally nilpotent derivations could have similar applications to the study of algebras that are not commutative.  We begin these investigations by considering domains of low Gelfand Kirillov dimension.  Given a finitely generated $k$-algebra $A$, the \emph{Gelfand-Kirillov dimension} (GK dimension for short) of $A$ is defined to be
$${\rm GKdim}(A) \ := \ \limsup_{n\to\infty} \frac{\log\, {\rm dim}(V^n)}{\log\,n},$$
where $V$ is a finite-dimensional $k$-vector subspace of $A$ that contains $1$ and generates $A$ as a $k$-algebra.  We note that this definition is independent of choice of vector space $V$ having the aforementioned properties.  In the case that $A$ is a finitely generated commutative algebra, the Gelfand-Kirillov dimension and the Krull dimension of $A$ are the same.  In this sense, the study of finitely generated noncommutative domains of GK dimension two can be seen as the noncommutative counterpart of the study of affine surfaces.  

Our main result is that we give a \emph{birational classification} of finitely generated domains of GK dimension $2$ that have a nonzero locally nilpotent derivation.  Given a domain $A$ of finite GK dimension, one can invert the nonzero elements of $A$ to form the \emph{quotient division ring} of $A$, which is denoted by $Q(A)$ and is a noncommutative analogue of the field of fractions of a commutative domain.  One of the deep problems in noncommutative algebra is to classify the quotient division rings of finitely generated domains of GK dimension $2$ up to isomorphism.  In analogy with classical algebraic geometry, such a classification is known as a \emph{birational classification}, and if the quotient division rings of two Ore domains $A$ and $B$ are isomorphic, we say that $A$ and $B$ are \emph{birationally isomorphic}.

In \S 2, we briefly describe Artin's proposed birational classification of graded domains of GK dimension $3$.  The classification of graded domains of GK dimension $3$ is intimately related to the classification of ungraded domains of GK dimension $2$, and is the noncommutative analogue of the Castelnuovo-Enriques birational classification of projective surfaces.  

One of the components of Artin's classification is formed by rings of differential operators on curves.  Given a smooth curve $X$ over an algebraically closed field $k$ of characteristic zero such that the ring of regular functions of $X$ is $A$, the  \emph{ring of differential operators} on $X$ is the subalgebra of ${\rm End}(A)$ formed by $A$ (regarded as $k$-linear endomorphisms coming from left multiplication) and by the $k$-linear derivations of $A$.  We denote this algebra by $\mathcal{D}(X)$.  Arguably, the most well-known example of a ring of differential operators is the Weyl algebra, which is the ring of differential operators on $X=\mathbb{A}^1(k)$.  In this case, the ring of regular functions is $k[x]$ and the derivations are a free $k[x]$-module spanned by the operator $y$, which is differentiation with respect to $x$.  This algebra has the relation $yx-xy=1$.  In the case that $X$ is not smooth, the definition is somewhat more technical, and we refer the reader to McConnell and Robson \cite{McR} for further reading.

Our first result gives a birational classification of finitely generated domains of GK dimension less than $3$ over an algebraically closed field of characteristic zero that have a nonzero locally nilpotent derivation.  We show that such rings are either subalgebras of the ring of  differential operators on a smooth irreducible affine curve or they satisfy a polynomial identity.  In the former case, we in fact show that the algebra is birationally isomorphic to the ring of differential operators on a curve.  The significance of satisfying a polynomial identity is that it implies that the quotient division algebra is finite-dimensional over its centre \cite[Theorem 13.6.5]{McR}. 
\begin{thm} \label{thm: main2} Let $k$ be an algebraically closed field of characteristic $0$ and let $A$ be a finitely generated $k$-algebra that is a domain of GK dimension strictly less than three.  If $A$ has a nonzero locally nilpotent derivation then $A$ either satisfies a polynomial identity or there is a smooth affine irreducible curve $X$ over $k$ such that:
\begin{enumerate}
\item $A$ is isomorphic to a subalgebra of $\mathcal{D}(X)$, the ring of differential operators on $X$;
\item $A$ is birationally isomorphic to $\mathcal{D}(X)$.
\end{enumerate}
\end{thm}
We note that the ring of differential operators on an irreducible affine algebraic curve always has ad-nilpotent elements coming from the elements of the ring of regular functions, and hence it always has non-trivial locally nilpotent derivations.

Unlike Krull dimension, GK dimension possesses the pathological property that in general it need not take integer values; indeed, for each real number $\alpha\ge 2$, there exists an algebra whose GK dimension is precisely $\alpha$ \cite[Theorem 3.14]{KL}.  Despite these pathological examples, there are no known examples of domains whose GK dimension is not either a nonnegative integer or infinite.  The second-named author showed that finitely generated connected graded domains of GK dimension at most $3$ have integer GK dimension.  As a corollary of Theorem \ref{thm: main2}, we obtain a similar gap theorem for domains that have a nonzero locally nilpotent derivation.
\begin{thm}\label{thm: main1} Let $k$ be an algebraically closed field of characteristic $0$ and let $A$ be a finitely generated $k$-algebra that is a domain with GK dimension in $[2,3)$.  If $A$ has a nonzero locally nilpotent derivation then $A$ has GK dimension $2$.
\end{thm}

In fact, in Theorem \ref{thm: main1} we are able to show that $A$ has \emph{quadratic growth}, which means that $A$ has a finite-dimensional subspace $V$ that contains $1$ and generates $A$ as an algebra and has the additional property that there exist positive constants $C_1$ and $C_2$ such that $C_1n^2 <{\rm dim}(V^n) < C_2n^2$ for all natural numbers.

The outline of this paper is as follows.  In \S2, we briefly describe Artin's proposed birational classification of noetherian domains of GK dimension $3$.  In \S3 we give the proofs of Theorems \ref{thm: main2} and \ref{thm: main1}.

\section{Artin's proposed birational classification}
In this section, we briefly describe Artin's \cite{Ar} proposed birational classification of a certain class of graded domains of GK dimension $3$ and its relation to Theorem \ref{thm: main2}.  We note that Artin warns the reader that ``everything should be taken with a grain of salt,'' as far as his proposed classification is concerned. 

To begin, we let $A$ be a graded noetherian domain of GK dimension $3$ that is generated in degree $1$ and we let $\mathcal{C}$ denote the category of finitely generated graded right $A$-modules modulo the subcategory of torsion modules.  (This can be thought of as the category of ``tails'' of finitely generated graded $A$-modules.)  We let ${\rm Proj}(A)$ denote the triple $(\mathcal{C},\mathcal{O},s)$, where $\mathcal{O}$ is the image of the right module $A$ in $\mathcal{C}$ and $s$ is the autoequivalence of $\mathcal{C}$ defined by the shift operator on graded modules.  

We note that $A$ has a \emph{graded quotient division ring}, denoted by $Q_{\rm gr}(A)$, which is formed by inverting the nonzero homogeneous elements of $A$.  Then there is a division ring $D$ and automorphism $\sigma$ of $D$ such that 
$$Q_{\rm gr}(A)\cong D[t,t^{-1};\sigma].$$  We think of $D$ as being the function field of $X={\rm Proj}(A)$.

Artin gives a proposed classification of the type of division rings $D$ that can occur when $A$ is a complex noetherian domain of GK dimension $3$.  (There are also a few other technical homological assumptions that he assumes the algebra possesses, but these are beyond the scope of this paper; we refer the interested paper to Artin's paper \cite{Ar}.)   If $A$ has these properties and $Q_{\rm gr}(A)\cong D[t,t^{-1};\sigma]$, the Artin asserts that up to isomorphism the possible division rings $D$ must satisfy at least one of the conditions on the list:
\begin{enumerate}
\item $D$ is finite-dimensional over its centre, which is a finitely generated extension of $\mathbb{C}$ of transcendence degree $2$;
\item $D$ is birationally isomorphic to a quantum plane; that is, $D$ is isomorphic to the quotient division of the complex domain generated by $x$ and $y$ with relation $xy=qyx$ for some nonzero $q\in \mathbb{C}$;
\item $D$ is isomorphic to the Skylanin division ring (see Artin \cite{Ar} for relevant definitions);
\item $D$ is isomorphic to the quotient division ring of the Weyl algebra;
\item $D$ is birationally isomorphic to $K[t;\sigma]$ or $K[t;\delta]$, where $K$ is a finitely generated field extension of $\mathbb{C}$ of positive genus and $\sigma$ and $\delta$ are respectively an automorphism and a derivation of $K$.   
\end{enumerate}
We note that the Weyl algebra and the rings of the form $K[t;\delta]$ with $K$ a finitely generated field extension of $\mathbb{C}$ and $\delta$ a derivation of $K$ are birationally isomorphic to a ring of differential operators of a smooth curve.  

We note that in many cases one can homogenize certain domains of GK dimension $2$ using a central indeterminate to obtain a graded domain of GK dimension $3$ satisfying the conditions that Artin assumes; indeed, many of the division rings on Artin's list are quotient division rings of noetherian domains of GK dimension $2$.  In this sense, there is a strong relationship between a birational classification of noetherian domains of GK dimension $2$ and graded noetherian domains of GK dimension $3$.  
\section{Classification theorem}
In this section we prove Theorems \ref{thm: main2} and \ref{thm: main1}.
We note that if $A$ is a Goldie domain with derivation $\delta$, then there is a natural extension of $\delta$ to the quotient division ring $Q(A)$ of $A$ via the quotient rule.  

We begin with a simple lemma.  We note that this lemma follows from the work of Bavula \cite[Lemma 2.1]{Ba}. 
\begin{lem} Let $k$ be a field of characteristic $0$ and let $A$ be a finitely generated $k$-algebra that is an Ore domain.  Suppose that $\delta :A \to A$ is a nonzero locally nilpotent derivation and let $\widehat{\delta}:Q(A) \to Q(A)$ be its unique extension to a derivation of the quotient division ring of $A$.  Then the kernel of $\widehat{\delta}$ is a division subalgebra $E$ of $Q(A)$ and $Q(A)$ is infinite-dimensional as a left $E$-vector space.
\label{lem: 1}
\end{lem}
\begin{proof} It is easily checked that the kernel of $\widehat{\delta}$ is a division ring and hence $E$ is a division subalgebra of $Q(A)$.  Since $\delta$ is nonzero, there exists some $a\in A$ that is not in $E$.  Let $m$ be the smallest natural number such that $\delta^m(a)\in E$.  Let $b=\delta^{m-1}(a)$.   Then $\delta(b)\in E$ is nonzero, as otherwise, $b\in {\rm ker}(\delta)\subseteq E$. 
Let $b_0=b \widehat{\delta}(b)^{-1}$.  Then $$\widehat{\delta}(b_0)=1.$$

Suppose that $Q(A)$ is finite-dimensional as a left $E$-vector space.  Then there exists some natural number $d$ and $e_d,\ldots ,e_0\in E$, with $e_0,\ldots ,e_d$ not all zero, such that
$$e_d b_0^d+e_{d-1}b_0^{d-1}+\cdots +e_0=0.$$
We may assume that $d$ is minimal with respect to having such a non-trivial relation.
Applying $\widehat{\delta}$ gives
$$de_d  b_0^{d-1}+\cdots + e_1 \ = \ 0.$$  By minimality of $d$, we have $e_d=\cdots = e_1=0$, but this immediately gives that $e_0=0$ as well, which is a contradiction.  The result follows.
\end{proof}
We next give a result which embeds $A$ in a skew polynomial extension in the case that $A$ has a nonzero locally nilpotent derivation, many parts of this lemma follow again from the work of Bavula \cite[Lemma 2.1]{Ba}, with the added ingredient coming from a result about subfields of division rings \cite{Be2}.
\begin{lem} Let $k$ be an algebraically closed field of characteristic $0$ and let $A$ be a finitely generated $k$-algebra that is a domain of GK dimension less than three.  Suppose that $\delta :A \to A$ is a nonzero locally nilpotent derivation and that $A$ does not satisfy a polynomial identity and let $E$ denote the kernel of the natural extension of $\delta$ to $Q(A)$.  Then $E$ is a field extension of $k$ of transcendence degree one and there exists some $y\in A\setminus E$ such that:
\begin{enumerate}
\item $[y,x]\in E$ for all $x\in E$;
\item $$A\subseteq \sum_{i\ge 0} Ey^i;$$ 
\item the algebra generated by $E$ and $y$ is isomorphic to $E[T;\mu]$ for some derivation $\mu$ of $E$ and with $y$ being sent to $T$ under the isomorphism;
\item after identifying $A$ with its image in $E[T;\mu]$ we have $Q(A)=Q(E[T;\mu])$.
\end{enumerate}
\label{lem: 2}
\end{lem}
\begin{proof} 
Note that $E$ is a division subalgebra of $D$.  Since $k$ is algebraically closed and $A$ has GK dimension less than three and does not satisfy a polynomial identity, we have that either $E$ is a field of transcendence degree at most $1$ over $k$ or $D$ is finite-dimensional as a left and right $E$-vector space \cite[Theorem 1.4]{Be2}.  Since $\delta$ is nonzero, we see that $E$ is a field by Lemma \ref{lem: 1}.   

Since $\delta$ is locally nilpotent, there exists some $y\in A$ such that $\delta(y)\not = 0$ and $\delta^2(y)=0$.  We necessarily have $\delta(y)\in E\setminus \{0\}$. Note that
if $x\in E$, then $\delta([y,x])=0$ and so $[y,x]\in E$ for all $x\in E$.  We let $b=\delta(y)^{-1}y\in Ey$.  Then $\delta(b)=1$.  Since $[y,E]\subseteq E$, we see that 
the subalgebra $B$ of $Q(A)$ generated by $E$ and $y$ is isomorphic to a homomorphic image of $E[T;\mu]$, where $\mu=[\, \cdot \, , y]$ is the derivation of $E$ constructed via commutation with $y$.  Note that $B$ is also generated by $E$ and $b$.  In fact, we see that $B\cong E[T;\mu]$, since otherwise there is some non-trivial relation $e_0+e_1b+\cdots +e_m b^m=0$ with $e_0,\ldots, e_m\in E$ and $m$ minimal.  Applying $\delta$ then gives
$$e_1+2e_2 b + \cdots +m e_m b^{m-1}=0,$$ contradicting the minimality of $m$.

We claim that $B$ contains $A$.  To see this, note that $\delta$ induces a locally nilpotent derivation of $B$ as $\delta(b)=1$ and $\delta(E)=\{0\}$.  Suppose that there is some $a\in A\setminus B$.  Then there exists some smallest natural number $m$ such that
$\delta^m(a)\in B$.  By replacing $a$ by $\delta^{m-1}(a)$, we may assume that $a\not\in B$ but $\delta(a)\in B$.  
Consider 
$$a_0:=\sum_{i=0}^{\infty} \frac{(-1)^i}{i!}\delta^{i+1}(a)b^i.$$  (Since $\delta$ induces a locally nilpotent derivation of $B$ and $\delta(a)\in B$, this sum is in fact finite and is an element of $B$.)
Observe that $$\delta(a_0) = \sum_{i=1}^{\infty} \frac{(-1)^i}{(i-1)!}\delta^{i+1}(a)b^{i-1} +
\sum_{i=0}^{\infty} \frac{(-1)^i}{i!}\delta^{i+2}(a)b^i = 
\delta(a).$$
Thus $a\in a_0+E\subseteq B$.  Finally, observe that $E$ and $T$ are in $Q(A)$ after identifying $A$ with its image in $E[T;\mu]$ and this gives (4).  The result follows.
\end{proof}
The next part of the proof involves showing that we can replace the field $E$ in the statement of Lemma \ref{lem: 2} with a finitely generated commutative $k$-algebra of Krull dimension one.  The first part of this proof is given by the following lemma.
\begin{lem} Let $k$ be a field, let $E$ be a field extension of $k$ of transcendence degree one, and let $\mu:E\to E$ be a $k$-algebra derivation of $E$.  If $A$ is a subalgebra of $E[T;\mu]$ of GK dimension strictly less than three that contains $T$ and $x,u\in (E\cap A)\setminus k$, then the sum $$\sum_{i=0}^{\infty} \mu^i(u)k[x]$$ is not direct.\label{lem: 3}
\end{lem}
\begin{proof} Suppose that $\sum_{i=0}^{\infty} \mu^i(u)k[x]$ is direct.  Let $V$ denote the $k$-vector subspace of $A$ spanned by $1,u,\mu(u),x,T$.  Since $A$ has GK dimension strictly, less than $3$, there is some $\epsilon>0$ such that ${\rm dim}(V^n)<n^{3-\epsilon}$ for all $n$ sufficiently large.
Note that $\sum_{i=0}^{\infty} ET^i$ is direct and $\mu^i(u)\in V^{i}$ for every $i$ and hence for each $n$, the set
$$\{\mu^i(u)x^jT^{\ell}~:~1\le i,j,\ell\le n\}$$ is a linearly independent subset of $V^{3n}$.  But this gives ${\rm dim}(V^{3n})\ge n^3$ for all $n$, a contradiction.
\end{proof}

\begin{prop} Let $k$ be an algebraically closed field of characteristic $0$ and let $A$ be a finitely generated $k$-algebra that is a domain of GK dimension less than three.  If $A$ has a nonzero locally nilpotent derivation $\delta$ then either $A$ satisfies a polynomial identity or there is a finitely generated commutative $k$-algebra $C$ of Krull dimension $1$ and a derivation $\mu$ of $C$ such that $A$ is isomorphic to a subalgebra of $C[T;\mu]$ and $A$ is birationally isomorphic to $C[T;\mu]$.
\label{prop: 1}
\end{prop}
\begin{proof}  Assume that $A$ does not satisfy a polynomial identity and let $D$ denote the quotient division ring of $A$.  Then $\delta$ extends to a derivation from $D$ to $D$.  Let $E\subseteq D$ denote the kernel of $D$.  

By Lemma \ref{lem: 2}, we have that $E$ is a field extension of $k$ of transcendence degree one and $A$ is isomorphic to a subalgebra of $E[T;\mu]$ that contains $T$.  We identify $A$ with its image in $E[T;\mu]$.  A straightforward argument shows that we can pick $x\in (A\cap E)\setminus k$, since otherwise, we could pick $a\in A$ such that $\delta(a)=1$ and could show by induction that every element in $A$ is a polynomial in $x$, contradicting the fact that $A$ does not satisfy a polynomial identity.   

By assumption $T\in A$, and we let
$z=[T,x]\in E\cap A$ and let $\gamma:A\to A$ denote the locally nilpotent derivation $[\, \cdot \, , x]$.

Suppose that $P(T)=c_0+\cdots + c_m T^m\in A$.  We claim that for each $i$, there is some natural number $n$ such that $c_i z^n\in A$.  To see this, suppose the claim is not true and let $j$ be the largest integer in $\{0,\ldots ,m\}$ for which the claim does not hold.  Then there is some $n$ such that $c_i z^n\in A$ for $i>j$.  Then $$Q(T):=z^n P(T) - \sum_{i>j} z^n c_i T^i = \sum_{i\le j} z^n c_i T^i\in A.$$

Note that $\gamma^j(Q(T)) =z^n c_j z^j = c_j z^{n+j}\in A$, contradicting the assumption that the claim does not hold for $c_j$.  

Let $P_1(T),\ldots ,P_m(T)\subseteq E[T;\mu]$ be generators for $A$ and let $S\subseteq E$ denote the set of all coefficients which occur in $P_1,\ldots ,P_m$.  Then we have shown that there is some $n$ such that
$$z^nS\subseteq A.$$

Let $B$ denote the subalgebra of $A$ generated by $z^nS, x, z$ and $T$.  Then $B$ has GK dimension strictly less than three.  Furthermore, by Lemma \ref{lem: 3} there exists some natural number $d$ such that
$$\mu^d(a)\in \sum_{i<d} \mu^i(a)k(x)$$ for all $a\in z^nS\cup \{x,z\}$.  It follows that there is some polynomial $f(x)$ such that
$$ \mu^d(a)f(x) \in \sum_{i<d} \mu^i(a)k[x]$$ for all $a\in z^nS\cup \{x,z\}$.
Let $C$ denote the commutative subalgebra of $E$ generated by
$1/z, 1/f(x)$ and $$\{\mu^i(a)~:~0\le i<d, a\in z^nS\cup \{x,z\}\}.$$
  
We first claim that $\mu$ restricts to a derivation of $C$.  To see this, note that
$$\mu(1/z)=-z^{-2}\mu(z)\in C,\qquad \mu(f(x)^{-1})=-f(x)^{-2} f'(x) z \in C,$$ and for $$a\in z^nS\cup \{x,z\}$$
and $0\le i<d$  we have
$\mu(\mu^i(a))=\mu^{i+1}(a) \in S$ if $i<d-1$ and if $i=d$, we have
$\mu(\mu^i(a))\in \sum_{j<d} \mu^j(a)k[x]f(x)^{-2}\subseteq C$. 

Let $c_1,\ldots, c_p$ denote the elements in the generating set $1/z, 1/f(x)$ and $$\{\mu^i(a)~:~0\le i<d, a\in z^nS\cup \{x,z\}\}.$$  Then given an element $c\in C$, we may write $c=g(c_1,\ldots ,c_p)$, for some polynomial $g(x_1,\ldots ,x_p)\in k[x_1,\ldots ,x_p]$.  Then $$\mu(c) = \sum_{i=1}^p \frac{\partial}{\partial x_i} g(x_1,\ldots ,x_p)|_{x_1=c_1,\ldots ,x_p=c_p} \cdot \mu(c_i) \in C.$$  Thus $\mu$ restricts to a derivation of $C$.  Furthermore $S\subseteq C$ and so $A\subseteq C[T;\mu]$.   

Since $A\subseteq C[T;\mu]\subseteq E[T;\mu]$ and $A$ and $E[T;\mu]$ have the same quotient division algebra, we see that $A$ and $C[T;\mu]$ have the same quotient division algebra.  The result follows.
\end{proof}
We observe that in this proof, we actually show that $k+z^nC[y;\mu]\subseteq A$ and so $A$ is sandwiched between an idealizer subring of a skew polynomial ring and a skew polynomial ring over a finitely generated commutative domain.
\begin{proof}[Proof of Theorems \ref{thm: main2} and \ref{thm: main1}]
By Proposition \ref{prop: 1} either $A$ satisfies a polynomial identity or there is a finitely generated commutative $k$-algebra $C$ of of Krull dimension $1$ and a derivation $\mu$ of $C$ such that $A$ is isomorphic to a subalgebra of $C[T;\mu]$ and $A$ is birationally isomorphic to $C[T;\mu]$.

To obtain Theorem \ref{thm: main1}, note that if $A$ satisfies a polynomial identity then it must have quadratic growth (cf. McConnell and Robson \cite[Proposition 13.10.6]{McR}).  If, on the other hand, $A$ is isomorphic to a subalgebra of $C[T;\mu]$ with $C$ a finitely generated algebra commutative algebra of Krull dimension $1$, then it follows from a result of Zhang \cite{Z} that any finitely generated subalgebra of $C[T;\mu]$ has GK dimension $2$.  In fact, if one follows Zhang's proof carefully, it can be checked that the growth is in fact quadratic.

To obtain Theorem \ref{thm: main2} note that there is an affine curve $X$ such that the ring of regular functions on $X$ is precisely $C$.  Since the singular set of $X$ is finite, $X$ has a smooth open affine subset $U$.   Furthermore, there is an affine open subset $Y$ of $U$ such that the module of derivations of $k[Y]$ is a free $k[Y]$-module of rank one (cf. Hartshorne \cite[Theorem 8.15]{Har}) and $Y$ is again a smooth irreducible affine curve.

Since derivations extend to localizations, we see $\mu$ extends naturally to $k[Y]$, which is a localization of $C$.  If we let $\delta_0$ denote a generator for the $k[Y]$-module of derivations of $k[Y]$, then we see that $\mu=c\delta_0$ for some nonzero $c\in k[Y]$.  Hence $A$ is isomorphic to a subalgebra of $\mathcal{D}(Y)\cong k[Y][T;\delta_0]$, the ring of differential operators on $Y$, and the two algebras are birationally isomorphic as we have only inverted elements of $C$ to obtain this ring of differential operators.  The result follows. 
\end{proof} 
\section{Concluding remarks and questions}
In this section, we make a few remarks and pose some related questions that we find interesting.

The first remark we make is that if $k$ is an algebraically closed field of characteristic $0$ and $A$ is a finitely generated $k$-algebra that is a domain of GK dimension strictly less than three, then we have shown that if $A$ has a nonzero locally nilpotent derivation then either $A$ satisfies a polynomial identity or there is a smooth irreducible affine curve $X$ and derivation $\mu:k[X]\to k[X]$ such that $A$ is isomorphic to a subalgebra of $k[X][T;\mu]$ that contains $T$.  We note that if we identify $A$ with its image in $k[X][T;\mu]$  and we pick $x\in (k[X]\cap A)\setminus k$ then $[x,T]=z\in k[X]$.  If we take the affine open subset $U$ of $X$ on which $z$ does not vanish, then $\mu$ extends to a derivation of $k[U]$ and 
$k[U][T;\mu]$ is a ring of quadratic growth and has elements $z^{-1}T,x$ whose commutator is $1$.  In particular, it has a subalgebra that is isomorphic to the Weyl algebra.

In light of this remark, we pose the following question.
\begin{quest} Let $k$ be an algebraically closed field of characteristic $0$ and let $A$ be a finitely generated $k$-algebra that is a domain of quadratic growth that is birationally isomorphic to a ring of differential operators on an affine curve over $k$.  Does there exist a finitely generated subalgebra $B$ of $Q(A)$ of quadratic growth that contains $A$ and has the property that the Weyl algebra is a subalgebra of $B$?
\end{quest}
We note that if $X$ is a smooth curve over an algebraically closed field of characteristic zero, then the quotient division ring of $\mathcal{D}(X)$ always contains a copy of the Weyl algebra.  We thus pose the following question.
\begin{quest} Can one classify all pairs $(X,Y)$ of smooth affine complex curves such that 
$\mathcal{D}(X)$ is isomorphic to a subalgebra of the quotient division ring of $\mathcal{D}(Y)$?
\end{quest}
We note that the Riemann-Hurwitz formula shows that if $X$ and $Y$ are smooth irreducible projective curves and $f:Y\to X$ is a surjective morphism, then the genus of $X$ is less than or equal to the genus of $Y$.  In light of this fact, we make the following conjecture.
\begin{conj} If $X$ and $Y$ are smooth curves such that $\mathcal{D}(X)$ is isomorphic to a subalgebra of the quotient division ring of $\mathcal{D}(Y)$, then the genus of $X$ is less than or equal to the genus of $Y$.
\end{conj}
We note that Theorem \ref{thm: main2} shows that if $k$ is an algebraically closed field of characteristic zero and $A$ is a finitely generated domain over $k$ of GK dimension less than three and that does not satisfy a polynomial identity and has a nonzero locally nilpotent derivation, then $A$ is a subalgebra of a ring of differential operators, which is noetherian.  Growth considerations show that this embedding of $A$ must be fairly tight, in some sense.  We thus make the following conjecture.
\begin{conj}  Let $k$ be an algebraically closed field of characteristic zero and let $A$ be a finitely generated $k$-algebra that is a domain of quadratic growth.  Suppose, in addition, that $A$ does not satisfy a polynomial identity and that $A$ has a nonzero locally nilpotent derivation.  Then $A$ is noetherian.
\end{conj}
One of the interesting phenomena that occurs in noncommutative algebra is that results for algebras of some GK dimension $d$ often have counterparts for graded algebras of GK dimension $d+1$.  We recall that in a graded Goldie domain $A$ one can invert the nonzero homogeneous elements to obtain the graded quotient division ring of $A$,  $$Q_{\rm gr}(A)\cong D[t,t^{-1};\sigma],$$ where $D$ is a division ring and $\sigma$ is an automorphism of $D$.   
\begin{quest} Let $A$ be a connected graded finitely generated complex domain of GK dimension $3$ and suppose that $A$ has a nonzero locally nilpotent derivation.  Then $Q_{\rm gr}(A)\cong D[t,t^{-1};\sigma]$ for some division ring $D$ and automorphism $\sigma$ of $D$.  Can one describe the set of $D$ that can occur under these hypotheses?  More generally, can one describe which pairs $(D,\sigma)$ can occur under these hypotheses?
\end{quest}
\section*{Acknowledgments}
We thank Tom Lenagan and Vladimir Bavula for many helpful comments and suggestions.

\end{document}